\documentclass[preprint,12pt]{elsarticle}
\usepackage{amsfonts}
\usepackage{amsmath}
\usepackage{mathrsfs}
\usepackage{amssymb}
\usepackage{bm}
\usepackage{amsthm}
\makeatletter
\def\ps@pprintTitle{%
   \let\@oddhead\@empty
   \let\@evenhead\@empty
   \let\@oddfoot\@empty
   \let\@evenfoot\@oddfoot
}
\makeatother

\def\q{{\mathcal Q}}

\def\u{\mathrm u}
\def\p{{\mathcal P}}
\def\I{{\mathcal I}}
\newcommand{\R}{{\mathbb{R}}}

\newcommand{\Z}{{\mathbb{Z}}}
\def\ddj{\dot \Delta_j}
\def\div{ \hbox{\rm div}\,  }

\newtheorem{theorem}{Theorem}[section]
\newtheorem{lemma}[theorem]{Lemma}

\newtheorem{remark}[theorem]{Remark}
\numberwithin{equation}{section}
\def\f{\frac}
\def\u{\bm{u}}

\let\wt=\widetilde

\begin{document}
\title{  Global large solutions  to the two dimensional compressible Navier-Stokes equations}

  \author{Xiaoping Zhai}
\ead{zhaixp@szu.edu.cn}

 \author{Zhi-Min Chen}
 \ead{zmchen@szu.edu.cn}
 \address{School  of Mathematics and Statistics, Shenzhen University, Shenzhen 518060, China}

\baselineskip=24pt

\begin{abstract}
We obtain the
global   large solutions to the compressible Navier-Stokes equations in $\mathbb{R}^2$. The  solution is large in the sense that there is no smallness assumption applied to  one component of the initial incompressible velocity.
\end{abstract}


\maketitle


\section{ Introduction and the main result}
The present paper is dedicated to the global wellposedness for the following compressible Navier-Stokes equations in $\R^2 $:
\begin{eqnarray}\label{mm}
\left\{\begin{aligned}
&\partial_t \rho + \div(\rho \u) = 0\, ,\\
 &     \partial_t ( \rho \u ) + \div ( \rho \u \otimes \u ) -\mu\Delta \u-(\lambda+\mu)\nabla\div \u+ \nabla P(\rho) =0,
 \\
 &\rho|_{t=0}=\rho_0, \,\, u|_{t=0}=\u_0\stackrel{\mathrm{def}}{=}(u^1_0,u^2_0),
\end{aligned}\right.
\end{eqnarray}
where
$\rho$  is the  density, $\u$  is the  velocity field,  $P(\rho)\in C^1$ is  the pressure field, which satisfies that   $P'(\rho)>0$ and
$P'(1)=1$.
The parameters  $\mu$ and $\lambda$   are  shear viscosity  and volume  viscosity coefficients respectively.

As one of the most popular fluid motion model in the field of the analysis and applications,  the compressible Navier-Stokes equations system  has attracted much attention and there is  a large  literature important to mathematical analysis and fluid mechanics.
One may mention in particular the works by Matsumura and Nishida  \cite{mat}, Xin \cite{Xin}, Lions \cite{lions+1}, Danchin \cite{danchin2000},  Hoff \cite{hoff1995},   Feireisl \cite{Feireisl1}, Feireisl {\it et al.}  \cite{ferisal} --\cite{Feireisl4}, Villani \cite{villani}, Chen {\it et.al} \cite{chenqionglei},  Charve and Danchin \cite{charve}, Haspot \cite{haspot},  Huang {\it et al.}
 \cite{huangxiangdi}, Danchin \cite{danchin2014}, \cite{danchin2018},   Kotschote \cite{Kotschote},  He {\it et.al} \cite{huangjingchi2}.
To summarize,
the most of  the previous global well-posedness results were established
by assuming that the density is close in some sense to a constant state and the initial velocity is  assumed to be small.
Here, we only recall part of  those results  in the critical Besov spaces.
By  studying the  behaviors of a hyperbolic-parabolic system, Danchin \cite{danchin2000} constructed the global small solutions of \eqref{m} in $L^2$ type Besov space.
With the aid of Green matrix,
  Chen {\it et al.} \cite{chenqionglei} and Charve, Danchin \cite{charve} obtained the global well-posedness result in the critical $L^p$
framework respectively.
 Later, Haspot \cite{haspot} gave a new proof via the so called  effective velocity.
 Wang {\it et al.} \cite{wangchao}
proved the global well-posedness of three dimensional compressible Navier-
Stokes equations for some classes of large initial data, which may have large oscillation
for the density and large energy for the velocity.
Recently,
Based on the dispersion property of acoustic waves,
Fang {\it et al.} \cite{fangdaoyuan} established the global strong solutions to \eqref{mm} in $\R^3$ which allows
 the low frequency part of the initial velocity field be large.
 He {\it et al.} \cite{huangjingchi2} also obtained the global solutions to \eqref{mm} in $\R^3$ with
 the large vertical component of the incompressible part of the initial velocity.
Very recently,  Danchin and Mucha \cite{danchin2018} proved the global solutions to \eqref{mm}  in $\R^d\,(d\ge2)$ with large initial velocity and almost constant density, if the  volume viscosity is large enough. This result was further extended by  Chen and Zhai \cite{zhaixiaoping} in a critical $L^p$  framework. It should be mentioned that there is no smallness restriction on  the incompressible part of the initial velocity in \cite{zhaixiaoping}  and \cite{danchin2018} in $\R^2.$ The proof of  \cite{zhaixiaoping}  and \cite{danchin2018} is based on the  global wellposedness for the two dimensional classical Navier-Stokes equations and  large  volume viscosity.

A natural question which arises is: when the volume viscosity is fixed, whether we can construct global solutions to the compressible Navier-Stokes equations in $\R^2$, without smallness restriction on  the incompressible part of the initial velocity? As far as we know, it's still an interesting open problem.  The main difficulty lie in that there has one bad term $\q(\p \u\cdot\nabla\u)$ (see the definition of $\p, \q$ in \eqref{touying}) in the equation of the compressible  part (see the second equation of \eqref{m3}). We cannot use perturbation method to disappear this term.
As an attempt, we  construct global solutions to the  compressible Navier-Stokes equations in $\R^2$ with the large  component of the  incompressible velocity with fixed   volume viscosity.
This is significant different to   \cite{zhaixiaoping} and \cite{danchin2018}  assuming the volume  viscosity being   sufficiently large.  The method used  here relies heavily on the algebraical
structure of the equation for the incompressible  part (see \eqref{m2}), that is to say, the equation of each component for the incompressible velocity  is a linear equation with coefficients depending on the density, the velocity of the compressible part and the other component of the incompressible velocity. Similar ideas have been used  in \cite{huangjingchi}, \cite{paicu2012} for inhomogeneous incompressible fluids. In the compressible case, it is much more complicated to deal with extra terms involved in $\q\u$.
For simplicity of notation,  we use the viscosity coefficient values  $ \mu=1 \,\,\,\mbox{ and } \lambda =0$
throughout the paper.


Because we shall seemingly consider the density functions as perturbations of the reference density 1, it's natural to set
   $a=\rho-1$, so that  system \eqref{mm} translates into
\begin{eqnarray}\label{m}
\left\{\begin{aligned}
&\partial_t a + \div  \u+ \div(a \u) = 0\, ,\\
 &     \partial_t \u  + \u\cdot \nabla \u -\Delta \u-\nabla\div  \u+ \nabla a = -L(a)\Delta \u- L(a)\nabla\div  \u
  +k(a)\nabla a
  ,\\
    &(a,\u)|_{t=0}=(a_0,\u_0),
\end{aligned}\right.
\end{eqnarray}
where
$$L(a)\stackrel{\mathrm{def}}{=}\frac{a}{1+a},\quad  k(a)\stackrel{\mathrm{def}}{=}-\frac{P'(1+a)}{1+a}+P'(1).$$
The present study lies on the  homogeneous Littlewood-Paley decomposition   $$z= \sum_{j \in \Z}\dot \Delta_j z \in\mathcal{S}'(\R^2),$$ which can be  truncated into  lower and higher oscillation parts expressed as
\begin{align}\label{eq:lhf0}
z^\ell\stackrel{\mathrm{def}}{=}\sum_{2^j\leq N_0}\ddj z\quad\hbox{and}\quad
z^h\stackrel{\mathrm{def}}{=}\sum_{2^   j>N_0}\ddj z
\end{align}
for a large  integer $N_0\ge 1.$  

In  the following, we  use the Leray projection operators
\begin{align}\label{touying}
\q =\nabla \Delta^{-1} \div,\,\,\, \p =\I-\q
\end{align}
 and  vector components
$$
v\stackrel{\mathrm{def}}{=}J_1(\p\u),\quad w\stackrel{\mathrm{def}}{=}J_2(\p\u),\quad v_0\stackrel{\mathrm{def}}{=}J_1(\p\u_0),\quad w_0\stackrel{\mathrm{def}}{=}J_2(\p\u_0),$$
with respect to $J_1\mathbf{M}=M^1$ is the first component of $\mathbf{M}=(M^1,M^2)$ and $J_2\mathbf{M}$ is the second component of $\mathbf{M}=(M^1,M^2).$

We are now in the position to state the main result of the present paper:
 \begin{theorem}\label{dingli}
Let   $a_0 \in \dot B^{0}_{2,1}\cap\dot B^{1}_{2,1}(\R^2), \u_0\in {\dot B^{0}_{2,1}}(\R^2)$.
Assume that there exist two positive constants $c_0$ and $C_0$ such that
\begin{align}\label{tiaojian}
&\|a_0 \|_{\dot B^{0}_{2,1}\cap\dot B^{1}_{2,1}}+\|(\q \u_0,{v}_0)\|_{\dot B^{0}_{2,1}}
\leq c_0\exp (-C_0(\|a_0 \|_{\dot B^{0}_{2,1}\cap\dot B^{1}_{2,1}}+\|(\q\u_0,{v}_0,{w}_0)\|_{\dot B^{0}_{2,1}})^2),
\end{align}
then the system \eqref{m}  admits a unique global
solution $(a,\u)$ satisfying,  for $t>0,$
\begin{align}\label{wrt1}
&\|(a,\q\u,{v})\|_{\wt L^\infty_t(\dot B^{0}_{2,1})}+\|a\|_{\wt L^\infty_t(\dot B^{1}_{2,1})}+\|a^h\|_{L^1_t(\dot B^{1}_{2,1})}
+\|(a^\ell,\q\u,{v})\|_{L^1_t(\dot B^{2}_{2,1})}\nonumber\\
&\quad\le C
(\|(a_0,\q\u_0,{v}_0)\|_{\dot B^{0}_{2,1}}+ \|a_0\|_{\dot B^{1}_{2,1}})
\exp (\|(a_0,\q\u_0,{v}_0,{w}_0)\|_{\dot B^{0}_{2,1}}+ \|a_0\|_{\dot B^{1}_{2,1}})^2,
\end{align}
and
\begin{align}\label{wrt2}
&\|{w}\|_{\wt L^\infty_t(\dot B^{0}_{2,1})}+\|{w}\|_{L^1_t(\dot B^{2}_{2,1})}\le C(\|(a_0,\q\u_0,{v}_0,{w}_0)\|_{\dot B^{0}_{2,1}}+ \|a_0\|_{\dot B^{1}_{2,1}})
\end{align}
for a  constant $C$.
  \end{theorem}

\begin{remark}
In \eqref{tiaojian}, there is no smallness restriction on  $w_0$, thus,
the above theorem improves the results of   \cite{chenqionglei} $ (p=2)$  and \cite{danchin2000} in $\R^2$.
\end{remark}

\section{Preliminaries } Denote  by $C$  a generic constant, which may vary from line to line.
Let us first recall  the theory of the Littlewood-Paley decomposition.
 The symbol $\mathcal{F}$ represents  the Fourier transform in $\mathcal{S}'(\R^2)$, the space of tempered distributions in $\R^2$. Let $\varphi$ be a nonnegative smooth function supported in an annulus  of $\R^2$ so that
\begin{align*}
 \sum_{j \in \mathbb{Z}} \varphi_j(\cdot)=1 \ \mbox{ in } \ \R^2\setminus \{0\} \ \mbox{ for } \,\,\varphi_j(\cdot)= \varphi(2^{-j}\cdot).
 \end{align*}
Therefore the homogeneous  dyadic blocks are defined as
$$\dot{\Delta}_ju=\mathcal{F}^{-1}( \varphi_j\mathcal{F}u).
$$
Denote by $\mathcal{S}_h^{'}({\mathbb R} ^2)$ the subspace of $u\in \mathcal{S}'(\R^2)$ such that
 the homogenous decomposition
$
u=\sum_{ j\in \mathbb{Z} }\dot{ \Delta }_{j}u
$
holds true. Hence we have the
homogeneous Besov space
 $$\dot{B}_{2,1}^s(\mathbb{R}^2)=\left\{u\in \mathscr{S}'_h(\mathbb{R}^2)\left|\, \|u\|_{\dot{B}_{2,1}^s}\stackrel{\mathrm{def}}{=} \sum_{j\in\Z}2^{js}\|\dot{\Delta}_ju\|_{L^2}<\infty\right.\right\} \mbox{ for } s \in \mathbb{R}.$$

For convenience, we use the symbols
\begin{align*}
&L^p_{T}(\dot{B}_{2,1}^s(\mathbb{R}^2))\stackrel{\mathrm{def}}{=} L^p(0,T; \dot{B}_{2,1}^s(\mathbb{R}^2)),\\
&{\widetilde{L}^p_{T}(\dot{B}_{2,1}^s(\mathbb{R}^2))}\stackrel{\mathrm{def}}{=} \left\{u: [0,T ]\mapsto \mathcal{S}'(\mathbb{R}^2)\left|\,
\|u\|_{\widetilde{L}^p_{T}(\dot{B}_{2,1}^s)}\stackrel{\mathrm{def}}{=} \sum_{j\in \mathbb{Z}}2^{js}
\|\dot{\Delta}_ju\|_{L^p(0,T;L^2(\R^2))}<\infty\right.\right\}.
\end{align*}


We will also repeatedly use the  following Bernstein inequality:
\begin{align}
&& C^{-1}\sigma ^k\|u\|_{L^q}\le\|\nabla^k u\|_{L^q}
\le C\sigma ^{k+\frac 2p-\frac 2q}\|u\|_{L^p} \mbox{ when } \mathrm{Supp} \,\mathcal{F}u\subset\sigma  \mathcal{C}\label{B2222}
\end{align}
for $1\le p \le q\le \infty$, $k\in \Z$,   $\mathcal{C}$ an annulus  of $\mathbb{R}^2$ and a constant $C$  independent of the scale parameter $\sigma >0$.

Moreover, we will use the following pointwise product law \cite[Lemma 2.7]{zhaixiaoping}
\begin{align}\label{daishu}
\|uv\|_{\dot{B}_{2,1}^{s_1+s_2 -1}}\lesssim \|u\|_{\dot{B}_{2,1}^{s_1}}\|v\|_{\dot{B}_{2,1}^{s_2}},\,\,\, s_1\leq 1, \,\, \,s_2\leq 1,\,\,\,s_1+s_2>0
\end{align}
and the  estimate \cite[Remark 2.102]{bcd}
\begin{align}\label{jiaohuanzi}
\sum_{j\in \mathbb{Z}}2^{js}\left\|[u\cdot \nabla ,\dot{\Delta}_j]v\right\|_{L^2}\lesssim  \|\nabla u\|_{\dot{B}_{2,1}^{1}}\|v\|_{\dot{B}_{2,1}^{s}}, \  s= 0,1
\end{align}
for the commutator $[u\cdot \nabla ,\dot{\Delta}_j]v= u\cdot \nabla \dot\Delta_jv - \dot\Delta_j (u\cdot \nabla v)$.
Here and in what follows, $a\lesssim b$ means the inequality  $a \le C b$ for a generic constant $C$.

Finally, we recall a composition estimate  \cite[Theorem 2.61]{bcd}
\begin{align}\label{fuhe}
\|F(f)\|_{\dot B^{s}_{2,1}}\lesssim\|f\|_{\dot B^s_{2,1}}, \,\,\, s>0,
\end{align}
 where  $F$ with $F(0)=0$ is  a smooth function defined on an open interval $I$ containing~$0$  and
 $f$  is valued in a bounded interval $A\subset I.$


\section{Proof of the main theorem}

The global solution is to be obtained by extending existing local solution with respect to time.

Given $a_0\in \dot{B}^{0}_{2,1}\cap\dot{B}^{1}_{2,1}(\R^2)$, $ \u_0\in
\dot{B}^{0}_{2,1}(\R^2)$ with  $\|a_0\|_{\dot{B}^{0}_{2,1}\cap\dot{B}^{1}_{2,1}}$ being sufficiently small, it follows from    \cite{danchin2014} that there exists a positive time $T$ so that \eqref{m} has a unique local solution $(a,\u)$ with
\begin{align}\label{local}
&a\in{C}((0,T]; \dot{B}^{0}_{2,1}\cap\dot{B}^{1}_{2,1}),\quad
 \u\in C((0,T];\dot{B}^{0}_{2,1})\cap L^1(0,T;\dot{B}^{2}_{2,1}).
\end{align}
Denote $T^*$ to be the largest time $T$ in  \eqref{local}. Hence to prove Theorem \ref{dingli}, we
only need to prove that $T^*=\infty$.
To do so, we need to produce  $a$ $priori$ estimates to the solution.

By using operators $\p$ and $\q$, we can decompose  the system \eqref{m} into two subsystems:
\begin{eqnarray}\label{m2}
\left\{\begin{aligned}
 & \partial_t{v}  +\u\cdot \nabla {v} + J_1([\p,\u\cdot \nabla] \u) -\Delta {v} = -J_1(\p(L(a)\Delta \u+ L(a)\nabla\div  \u)),\\
 &\partial_t{w}  +\u\cdot \nabla {w} + J_2([\p,\u\cdot \nabla] \u) -\Delta {w} = -J_2(\p(L(a)\Delta \u+ L(a)\nabla\div  \u)),\\
\end{aligned}\right.
\end{eqnarray}
and
\begin{eqnarray}\label{m3}
\left\{\begin{aligned}
&\partial_t a+\u\cdot \nabla a +\div \q \u=-a \div  \u,\\
&\partial_t \q \u +\u\cdot \nabla \q \u-\Delta \q \u-\nabla \div \q \u+\nabla a= \q G,
\end{aligned}\right.
\end{eqnarray}
with
 \begin{align*}
  \q G=&-[\q ,\u\cdot \nabla] \u-\q (L(a)\Delta \u)-\q ( L(a)\nabla\div  \u)+\q (k(a)\nabla a).
  \end{align*}

\begin{lemma}\label{zhongyao}
Let $(a,\u)$ be the global smooth solution to \eqref{m}, then there holds the following inequality:
\begin{align*}
\| \q G\|_{\dot B^{0}_{2,1}}
\lesssim&\|{v}\|_{\dot B^{1}_{2,1}}\|{w}\|_{\dot B^{1}_{2,1}}\nonumber\\
&+(\|(a,{v},\q \u)\|_{\dot B^{0}_{2,1}}+\|a\|_{\dot B^{1}_{2,1}} )(\|a^h\|_{\dot B^{1}_{2,1}}+\|(a^\ell,{v},{w},\q \u)\|_{\dot B^{2}_{2,1}}).
\end{align*}

\end{lemma}

\begin{proof}
It follows from $\u=\p \u+\q \u$ that
\begin{align}\label{Q2}
[\q, \u\cdot \nabla ]\u&=[\q, (\q \u+\p \u)\cdot\nabla ](\q \u+\p \u)\nonumber\\
&=[\q, (\q \u)\cdot\nabla ]\p \u+[\q, (\q \u)\cdot\nabla ]\q \u\nonumber\\
&\quad+[\q, (\p \u)\cdot\nabla ]\q \u+[\q, (\p \u)\cdot\nabla ]\p \u.
\end{align}
Applying \eqref{daishu}, we have
\begin{align}\label{Q3}
&\big\|[\q, (\q \u)\cdot\nabla ]\p \u
\big\|_{\dot B^{0}_{2,1}} \lesssim\big\|\q ((\q \u)\cdot\nabla \p \u)
\big\|_{\dot B^{0}_{2,1}}
\lesssim\|({v},{w})\|_{\dot B^{2}_{2,1}}   \|\q \u\| _{\dot B^{0}_{2,1}},\nonumber\\
&\big\|[\q, (\q \u)\cdot\nabla ]\q \u\big\|_{\dot B^{0}_{2,1}}\lesssim \big\| (\q \u)\cdot\nabla \q \u\big\|_{\dot B^{0}_{2,1}} \lesssim \|  \q \u\|_{\dot B^{0}_{2,1}}\|  \q \u\|_{\dot B^{2}_{2,1}}.
\end{align}
It's not  hard to check
\begin{align*}
\ddj([\q, (\p \u)\cdot\nabla ]\q \u)=[\ddj\q, (\p \u)\cdot\nabla ]\q \u-[\ddj,(\p \u)\cdot\nabla ]\q \u.
\end{align*}

Thus, from   \eqref{jiaohuanzi}, one infer that
\begin{align}\label{Q8}
\big\|[\q, (\p \u)\cdot\nabla ]\q \u    \big\|_{\dot B^{0}_{2,1}}
\lesssim&\|({v},{w})\|_{\dot B^{2}_{2,1}}   \|\q \u\| _{\dot B^{0}_{2,1}}.
\end{align}
Thanks to   the divergence free property $\partial_2{w}=-\partial_1{v}$, we  deduce
\begin{align}\label{Q5}
\big|[\q, (\p \u)\cdot\nabla ]\p \u\big|
&\lesssim\big| (\p \u)\cdot\nabla \p \u\big|\nonumber\\
&\lesssim\big| (\p \u)\cdot\nabla {v}\big|+\big| (\p \u)\cdot\nabla {w}\big|\nonumber\\
&\lesssim\big| {w}\partial_1{v}-{v} \partial_2{w}\big|+\big|{v}\partial_1{w}-{w} \partial_1{v}\big|\nonumber\\
&\lesssim\big| {v}\partial_2 {w}\big|+\big| {w}\partial_2 {v}\big|+\big|{v}\partial_1{w}\big|+\big|{w} \partial_1{v}\big|.
\end{align}
By using \eqref{daishu} again, we have
\begin{align}\label{Q6}
&\big\|{v}\partial_2{w}\big\|_{\dot B^{0}_{2,1}}+\big\|{v}\partial_1{w}\big\|_{\dot B^{0}_{2,1}}
\lesssim \|{v}\|_{\dot B^{0}_{2,1}}\|{w}\|_{\dot B^{2}_{2,1}},
\nonumber\\
&\big\| {w}\partial_2 {v}\big\|_{\dot B^{0}_{2,1}}+\big\| {w}\partial_1 {v}\big\|_{\dot B^{0}_{2,1}}
\lesssim \|{v}\|_{\dot B^{1}_{2,1}}\|{w}\|_{\dot B^{1}_{2,1}}.
\end{align}

The combination of \eqref{Q3}--\eqref{Q6} yields
\begin{align}\label{Q11}
\big\|[\q, \u\cdot \nabla ]\u   \big\|_{\dot B^{0}_{2,1}} \lesssim&\|{v}\|_{\dot B^{1}_{2,1}}\|{w}\|_{\dot B^{1}_{2,1}}
  \nonumber\\
  &+(\|\q \u\| _{\dot B^{0}_{2,1}}+\|{v}\|_{\dot B^{0}_{2,1}})(\|\q \u\| _{\dot B^{2}_{2,1}}+\|({v},{w})\| _{\dot B^{2}_{2,1}}).
\end{align}

Throughout we make the assumption that
\begin{equation}\label{axiao}
\sup_{t\in\R_+,\, x\in\R^2} |a(t,x)|\leq \frac12
\end{equation}
which will enable us to use freely the composition estimate stated in \eqref{fuhe}.
Note that as $\dot B^{1}_{2,1}(\R^2)\hookrightarrow L^\infty(\R^2),$ condition \eqref{axiao} will be ensured by the
fact that the constructed solution about $a$ has small norm.

Thus, one can obtain
\begin{align}\label{A12}
\| \q G\|_{\dot B^{0}_{2,1}}
\lesssim&\|[\q, \u\cdot \nabla ]\u\|_{ \dot B^{0}_{2,1} }+\|L(a)\Delta \u+ L(a)\nabla\div  \u \|_{ \dot B^{0}_{2,1} }+\|k(a)\nabla a\|_{ \dot B^{0}_{2,1} }\nonumber\\
\lesssim&\|{v}\|_{\dot B^{1}_{2,1}}\|{w}\|_{\dot B^{1}_{2,1}}
  +\| L(a) \|_{ \dot B^{1}_{2,1} }\|\mathbf{u} \|_{ \dot B^{2}_{2,1} }
+\| k(a)\|_{ \dot B^{1}_{2,1} }\| a\|_{ \dot B^{1}_{2,1} }\nonumber\\
&+
\|({v},\q \u)\|_{\dot B^{0}_{2,1}})(\|\q \u\| _{\dot B^{2}_{2,1}}+\|({v},{w})\| _{\dot B^{2}_{2,1}}) \nonumber\\
\lesssim&\|{v}\|_{\dot B^{1}_{2,1}}\|{w}\|_{\dot B^{1}_{2,1}}\nonumber\\
&+(\|a\|_{\dot B^{1}_{2,1}}+\|(a,{v},\q \u)\|_{\dot B^{0}_{2,1}} )(\|a^h\|_{\dot B^{1}_{2,1}}+\|(a^\ell,{v},{w},\q \u)\|_{\dot B^{2}_{2,1}}),
\end{align}
in which we have used the  fact:
\begin{align*}
\| a\|_{ \dot B^{1}_{2,1} }^2\lesssim&\| a^\ell\|_{ \dot B^{1}_{2,1} }^2+\| a^h\|_{ \dot B^{1}_{2,1} }^2\lesssim\| a^\ell\|_{ \dot B^{0}_{2,1} }\| a^\ell\|_{ \dot B^{2}_{2,1} }+\| a^h\|_{ \dot B^{1}_{2,1} }^2\nonumber\\
\lesssim&(\| a\|_{ \dot B^{0}_{2,1} }+\| a\|_{ \dot B^{1}_{2,1} })(\| a^\ell\|_{ \dot B^{2}_{2,1} }+\| a^h\|_{ \dot B^{1}_{2,1} }).
\end{align*}
This  completes  the proof of  the Lemma \ref{zhongyao}.
\end{proof}


Now, we present the energy estimates for $a$, $\q\u$, ${v}$ and  ${w}$ respectively in the framework of Besov spaces.

Applying $\ddj $ to the first equation in \eqref{m3} gives
\begin{align}\label{dijiu1}
\partial_t \ddj a+\u\cdot \nabla \ddj a +\div \q \ddj \u+[\ddj,\u\cdot\nabla] a=-\ddj (a\, \div  \u).
\end{align}
Taking $L^2$ inner product of $\ddj a$ with   \eqref{dijiu1} and using integrating by parts, we have
\begin{align}\label{dijiu2}
&\f12\f{d}{dt}\|{\ddj a}\|_{L^2}^2+\int_{\R^2} {\ddj a} \cdot\div {\ddj \q\u}\,dx\nonumber\\
&\quad=\f12\int_{\R^2} \div \u |{\ddj a}|^2\,dx-\int_{\R^2}[\ddj,\u\cdot\nabla] a\cdot  {\ddj a}\,dx-\int_{\R^2} \ddj(a\div \u) \cdot {\ddj a}\,dx.
\end{align}
Applying $\ddj $ to the second equation in \eqref{m3} gives
\begin{align}\label{dijiu3}
\partial_t \ddj\q \u + \u\cdot\nabla \ddj\q \u -2\Delta \ddj\q \u +\nabla \ddj a= \ddj\q G-[\ddj,\u\cdot\nabla] \q \u.
\end{align}
We can get by using a similar derivation of \eqref{dijiu2} that
\begin{align}\label{dijiu4}
&\f12\f{d}{dt}\|{\ddj \q\u}\|_{L^2}^2+2\|\nabla{\ddj \q\u}\|_{L^2}^2-\int_{\R^2} {\ddj a} \cdot\div {\ddj \q\u}\,dx\nonumber\\
&\quad=\f12\int_{\R^2} \div \u |{\ddj \q\u}|^2\,dx\nonumber\\
&\quad\quad-\int_{\R^2}[\ddj,\u\cdot\nabla] \q \u\ \cdot {\ddj \q\u}\,dx+\int_{\R^2}{\ddj \q G} \cdot{\ddj \q\u}\,dx.
\end{align}
Applying  the gradient $\nabla$ on  \eqref{dijiu1}, we have
\begin{align}\label{dijiu5}
\partial_t \nabla {\ddj a} +\u\cdot\nabla \nabla {\ddj a}+\nabla\div {\ddj \q\u}=F_j(a,\u)
\end{align}
with
$
F_j(a,\u)\stackrel{\mathrm{def}}{=} -\nabla([\ddj,\u\cdot\nabla]  a)-\nabla \ddj(a\div \u)-\nabla \u\cdot\nabla {\ddj a}.
$

Taking $L^2$ inner product of  $\ddj \nabla a$ with the previous equation  gives
\begin{align}\label{dijiu6}
&\f{d}{dt}\|\nabla {\ddj a}\|_{L^2}^2+2\int_{\R^2} \nabla {\ddj a} \cdot\nabla\div {\ddj \q\u}\,dx\nonumber\\
&\quad=\int_{\R^2} \div \u |\nabla {\ddj a}|^2\,dx+2\int_{\R^2}F_j(a,\u)\  \cdot\nabla {\ddj a}\,dx.
\end{align}
Testing  \eqref{dijiu3}  by $\nabla \ddj a$ and \eqref{dijiu5} by $\ddj \q\u$,  we get
\begin{align}\label{dijiu7}
&\frac{d}{dt}\int_{\R^2}{\ddj \q\u} \cdot\nabla {\ddj a}\, \,dx+\int_{\R^2}|\nabla {\ddj a}|^2\, \,dx\nonumber\\
&\quad\quad+\int_{\R^2} \nabla\div {\ddj \q\u} \cdot{\ddj \q\u} \,dx-2\int_{\R^2}\Delta {\ddj \q\u}\cdot\nabla {\ddj a}\,dx\nonumber\\
&\quad=\int_{\R^2} F_j(a,\u) \cdot{\ddj \q\u}\,dx+\int_{\R^2} {\ddj \q\u}\cdot\nabla {\ddj a}\ \div \u\,dx
\nonumber\\
&\quad\quad+\int_{\R^2} \big({\ddj \q G}-[\ddj,\u\cdot\nabla] \q \u\big)\cdot\nabla {\ddj a}\,dx.
\end{align}
in which we have used the following equality:
$$
\int_{\R^2} \u\cdot\nabla ({\ddj \q\u}\cdot\nabla {\ddj a})\,dx = -\int_{\R^2} {\ddj \q\u}\cdot\nabla {\ddj a}\ \div \u\,dx.
$$
Denote
\begin{equation*}
 L^2_j \stackrel{\mathrm{def}}{=} \int_{\R^2}  \bigl(|{\ddj a}|^2 + |{\ddj \q\u}|^2 + 2 {\ddj \q\u} \cdot\nabla {\ddj a} + 2|\nabla {\ddj a}|^2)\,\,dx.
\end{equation*}

Combining with estimates \eqref{dijiu2}, \eqref{dijiu4}, \eqref{dijiu6} and \eqref{dijiu7}, we get
\begin{align}\label{dijiu9}
&\frac12\frac{d}{dt} L^2_j+ \int_{\R^2} (|\nabla{\ddj \q\u}|^2 + |\nabla {\ddj a}|^2)\,\,dx\nonumber\\
&\quad=\f12\int_{\R^2} \div \u |{\ddj a}|^2\,dx-\int_{\R^2}[\ddj,\u\cdot\nabla] a\cdot  {\ddj a}\,dx+\int_{\R^2} \ddj(a\div \u) \cdot {\ddj a}\,dx\nonumber\\
&\quad\quad+\f12\int_{\R^2} \div \u |{\ddj \q\u}|^2\,dx-\int_{\R^2}[\ddj,\u\cdot\nabla] \q \u\ \cdot {\ddj \q\u}\,dx+\int_{\R^2}{\ddj \q G} \cdot{\ddj \q\u}\,dx\nonumber\\
&\quad\quad+\int_{\R^2} \div \u |\nabla {\ddj a}|^2\,dx+2\int_{\R^2}F_j(a,\u)\  \cdot\nabla {\ddj a}\,dx+\int_{\R^2} F_j(a,\u) \cdot{\ddj \q\u}\,dx\nonumber\\
&\quad\quad+\int_{\R^2} {\ddj \q\u}\cdot\nabla {\ddj a}\ \div \u\,dx+\int_{\R^2} \big({\ddj \q G}-[\ddj,\u\cdot\nabla] \q \u\big)\cdot\nabla {\ddj a}\,dx.
\end{align}

It is  readily seen  that
\begin{equation*}
L_j\approx\|({\ddj \q\u},{\ddj a},2\nabla {\ddj a})\|_{L^2}\quad\hbox{ for all }\ j\in\Z,
\end{equation*}
\begin{equation*}
\int_{\R^2} ( |\nabla {\ddj \q\u}|^2 + |\nabla {\ddj a}|^2)\, \,dx \geq c\min(2^{2j},2^{-2})L_j^2
\end{equation*}
for a constant $c$.
Therefore, we deduce from  \eqref{dijiu9} that
\begin{align*}
&\frac12\frac{d}{dt} L_j ^2+ c\min(2^{2j},2^{-2}) L_j^2\nonumber\\
&\quad\lesssim
\|\nabla \u\|_{L^\infty}L_j^2
+ \Big(\|(\ddj(a\div \u),{\ddj \q G})\|_{L^2} \nonumber\\
 &\quad\quad+\|F_{j}(a,\u)\|_{L^2}+\| [\ddj,\u\cdot\nabla] a\|_{L^2}+\| [\ddj,\u\cdot\nabla] \q \u\|_{L^2}\Big)L_j.
\end{align*}
Hence integrating in time and using the definition of $F_j(a,\u)$, we can finally get
\begin{align}\label{dijiu14}
& L_j(t)+ c\min(2^{2j},2^{-2}) \int_0^tL_j\,d\tau\nonumber\\
 &\quad\lesssim L_j(0)
+\int_0^t\|\nabla \u\|_{L^\infty}L_j\,d\tau+ \int_0^t\Big(\|(\ddj(a\div \u), \ddj\nabla(a\div \u),{\ddj \q G})\|_{L^2} \nonumber\\
 &\quad\quad+\| \nabla([\ddj,\u\cdot\nabla]  a)\|_{L^2}+\| [\ddj,\u\cdot\nabla] a\|_{L^2}+\| [\ddj,\u\cdot\nabla] \q \u\|_{L^2}\Big)\,d\tau.
\end{align}
From estimate \eqref{dijiu4}, we  get
\begin{align*}
&\f12\f{d}{dt}\|{\ddj \q\u}\|_{L^2}^2+2^{2j+1}\|{\ddj \q\u}\|_{L^2}^2\nonumber\\
&\quad\lesssim
\|\nabla \u\|_{L^\infty}\|{\ddj \q\u}\|_{L^2}^2+(\|\nabla {\ddj a}\|_{L^2}+\|[\ddj,\u\cdot\nabla] \q \u\|_{L^2}+\|{\ddj \q G}\|_{L^2})\|{\ddj \q\u}\|_{L^2},
\end{align*}
which implies
\begin{align}\label{dijiu15}
&\|{\ddj \q\u}(t)\|_{L^2}+\int_0^t2^{2j+1}\|{\ddj \q\u}\|_{L^2}\,d\tau\nonumber\\
&\quad\lesssim \|{\ddj \q\u}(0)\|_{L^2}+\int_0^t
\|\nabla \u\|_{L^\infty}\|{\ddj \q\u}\|_{L^2}\,d\tau\nonumber\\
&\quad\quad+\int_0^t(\|\nabla {\ddj a}\|_{L^2}+\|[\ddj,\u\cdot\nabla] \q \u\|_{L^2}+\|{\ddj \q G}\|_{L^2})\,d\tau.
\end{align}
This together with \eqref{dijiu14} yields
\begin{align}\label{dijiu16}
&\|(a,\q \u)(t)\|_{\dot B^{0}_{2,1}}+\|a(t)\|_{\dot B^{  1}_{2,1}}
+\int^t_0\big(\|(a^\ell,\q \u)\|_{\dot B^{2}_{2,1}}+\|a^h\|_{\dot B^{1}_{2,1}}\big)\,d\tau
\nonumber\\
&\quad\lesssim\|(a_0,\q \u_0)\|_{\dot{B}_{2,1}^{0}}+\|a_0\|_{\dot{B}_{2,1}^{1}}+\int_0^t\|\nabla \u\|_{L^\infty}\|(a,\nabla a,\q \u)\|_{\dot{B}_{2,1}^{0}}\,d\tau\nonumber\\
&\quad\quad+ \int_0^t(\| a\div \q \u   \|_{\dot B^{0}_{2,1}}+\| \nabla(a\div \q \u )  \|_{\dot B^{0}_{2,1}}+\| \q G\|_{\dot B^{0}_{2,1}})\,d\tau\nonumber\\
&\quad\quad+ \int_0^t\sum_{j\in\Z}\big(\| [\ddj,\u\cdot\nabla] a\|_{L^2}+\| \nabla([\ddj,\u\cdot\nabla]  a)\|_{L^2}+\| [\ddj,\u\cdot\nabla] \q \u\|_{L^2}\big)\,d\tau.
\end{align}
By using estimate \eqref{daishu}, we have
\begin{align}\label{dijiu17}
\| a\div \q \u   \|_{\dot B^{0}_{2,1}}
\lesssim&\|a\|_{\dot B^{0}_{2,1}}\|\div \q \u\|_{\dot B^{1}_{2,1}}
\lesssim\|a\|_{\dot B^{0 }_{2,1}}\|\q \u\| _{\dot B^{2}_{2,1}},\nonumber\\
\| \nabla(a\div \q \u )  \|_{\dot B^{0}_{2,1}}
\lesssim&\|a\div \q \u  \|_{\dot B^{1}_{2,1}}
\lesssim\|a\|_{\dot B^{1}_{2,1}}\|\div \q \u\|_{\dot B^{1}_{2,1}}
\lesssim\|a\|_{\dot B^{1}_{2,1}}\| \q \u\|_{\dot B^{2}_{2,1}}
.
\end{align}
The terms involved in the commutator can be obtained from
\eqref{jiaohuanzi} that
\begin{align}\label{dijiu18}
\sum_{j\in\Z} \| [\ddj ,\u\cdot \nabla ] a\|_{L^2}&\lesssim \sum_{j\in\Z} \| [\ddj ,\p \u\cdot \nabla ] a\|_{L^2}+\sum_{j\in\Z} \| [\ddj ,\q \u\cdot \nabla ] a\|_{L^2}\nonumber\\
&\lesssim  \|({v},{w},\q \u)\|_{\dot B^{2}_{2,1}} \|a\|_{\dot B^{0}_{2,1}},
\end{align}
\begin{align}\label{dijiu19}
\sum_{j\in\Z} \|\nabla( [\ddj ,\u\cdot \nabla ]  a)\|_{L^2}&\lesssim \sum_{j\in\Z} \| \nabla([\ddj ,\p \u\cdot \nabla ]  a)\|_{L^2}+\sum_{j\in\Z} \| \nabla([\ddj ,\q \u\cdot \nabla ]  a)\|_{L^2}\nonumber\\
&\lesssim  \|({v},{w},\q \u)\|_{\dot B^{2}_{2,1}} \|a\|_{\dot B^{1}_{2,1}},
\end{align}
\begin{align}\label{B15}
\sum_{j\in\Z} \| [\ddj ,\u\cdot \nabla ] \q \u\|_{L^2}&\lesssim \sum_{j\in\Z} \| [\ddj ,\p \u\cdot \nabla ] \q \u\|_{L^2}+\sum_{j\in\Z} \| [\ddj ,\q \u\cdot \nabla ] \q \u\|_{L^2}\nonumber\\
&\lesssim  \|({v},{w},\q \u)\|_{\dot B^{2}_{2,1}} \|\q \u\|_{\dot B^{0}_{2,1}}.
\end{align}
Inserting  \eqref{dijiu17}--\eqref{B15} into  \eqref{dijiu16}  and using Lemma \ref{zhongyao}, we have
\begin{align}\label{A19}
&\|(a,\q \u)(t)\|_{\dot B^{0}_{2,1}}+\|a(t)\|_{\dot B^{  1}_{2,1}}
+\int^t_0\big(\|(a^\ell,\q \u)\|_{\dot B^{2}_{2,1}}+\|a^h\|_{\dot B^{1}_{2,1}}\big)\,d\tau
\nonumber\\
&\quad\lesssim\|(a_0,\q \u_0)\|_{\dot{B}_{2,1}^{0}}+\|a_0\|_{\dot{B}_{2,1}^{1}}\nonumber\\
&\quad\quad+\int_0^t(\|a\|_{\dot B^{1}_{2,1}}+\|(a,{v},\q \u)\|_{\dot B^{0}_{2,1}} )(\|a^h\|_{\dot B^{1}_{2,1}}+\|(a^\ell,{v},{w},\q \u)\|_{\dot B^{2}_{2,1}})\,d\tau.
\end{align}
Applying $\ddj$ to  the  first equation  in \eqref{m2} and using a standard energy argument, we have
\begin{align}\label{D12}
&\|{v}\|_{\widetilde{L}^\infty_t(\dot B^{0}_{2,1})}+\|{v}\|_{L^1_t(\dot B^{2}_{2,1})}\nonumber \\
&\quad\lesssim\|{v}_0\|_{\dot B^{0}_{2,1}}+ \int_0^t\|\div  \u\|_{L^\infty}\|{v}\|_{ \dot B^{0}_{2,1} }\,d\tau+ \int_0^t\sum_{j\in\Z}\| [\ddj,\u\cdot\nabla] v\|_{L^2}\,d\tau\nonumber \\
&\quad\quad+\int_0^t\|\p (L(a)\Delta \u+L(a)\nabla\div \u)\|_{ \dot B^{0}_{2,1} }\,d\tau +\int_0^t\|[\p,\u\cdot \nabla] \u\|_{ \dot B^{0}_{2,1} }\,d\tau.
\end{align}
Using the embedding relation $ \dot B^{1}_{2,1}(\R^2)\hookrightarrow L^\infty(\R^2)$, we get
\begin{align}\label{D12+1}
\|\div  \u\|_{L^\infty}\|{v}\|_{ \dot B^{0}_{2,1} }\lesssim\|\div \u \|_{ \dot B^{1}_{2,1} }\|{v} \|_{ \dot B^{0}_{2,1} }\lesssim\|\q \u \|_{ \dot B^{2}_{2,1} }\|{v} \|_{ \dot B^{0}_{2,1} }.
\end{align}
The term about commutator can be estimated similarly to \eqref{dijiu18} that
\begin{align}\label{dijiu1890}
\sum_{j\in\Z} \| [\ddj ,\u\cdot \nabla ] v\|_{L^2}
&\lesssim  \|({v},{w},\q \u)\|_{\dot B^{2}_{2,1}} \|v\|_{\dot B^{0}_{2,1}}.
\end{align}
Thanks to \eqref{daishu} and \eqref{fuhe}, we deduce that
\begin{align}\label{D12+2}
\|\p (L(a)\Delta \u+L(a)\nabla\div \u)\|_{ \dot B^{0}_{2,1} }
\lesssim&\|L(a)\|_{ \dot B^{1}_{2,1} }\|\nabla^2 \u \|_{ \dot B^{0}_{2,1} }\nonumber\\
\lesssim&\|a \|_{ \dot B^{1}_{2,1} }\|({v},{w},\q \u) \|_{ \dot B^{2}_{2,1} }.
\end{align}

In order to deal with the last term in \eqref{D12}, we deduce from
 $\u=\p \u+\q \u$ that
\begin{align}\label{QWER2}
[\p, \u\cdot \nabla ]\u&=[\p, (\q \u+\p \u)\cdot\nabla ](\q \u+\p \u)\nonumber\\
&=[\p, (\q \u)\cdot\nabla ]\p \u+[\p, (\q \u)\cdot\nabla ]\q \u\nonumber\\
&\quad+[\p, (\p \u)\cdot\nabla ]\q \u+[\p, (\p \u)\cdot\nabla ]\p \u.
\end{align}
Applying \eqref{daishu}, we have
\begin{align}\label{QWER3}
&\big\|[\p, (\q \u)\cdot\nabla ]\p \u
\big\|_{\dot B^{0}_{2,1}} \lesssim\big\|(\q \u)\cdot\nabla \p \u
\big\|_{\dot B^{0}_{2,1}}
\lesssim\|({v},{w})\|_{\dot B^{2}_{2,1}}   \|\q \u\| _{\dot B^{0}_{2,1}},\nonumber\\
&\big\|[\p, (\q \u)\cdot\nabla ]\q \u\big\|_{\dot B^{0}_{2,1}}\lesssim \big\| (\q \u)\cdot\nabla \q \u\big\|_{\dot B^{0}_{2,1}} \lesssim \|  \q \u\|_{\dot B^{0}_{2,1}}\|  \q \u\|_{\dot B^{2}_{2,1}}.
\end{align}
By virtue of
\begin{align}\label{QWER7}
\ddj([\p, (\p \u)\cdot\nabla ]\q \u)=[\ddj\p, (\p \u)\cdot\nabla ]\q \u,
\end{align}
and  an argument  similar to the  derivation of    \eqref{Q8}, we have
\begin{align}\label{QWER8}
\big\|[\p, (\p \u)\cdot\nabla ]\q \u    \big\|_{\dot B^{0}_{2,1}}
\lesssim&\|({v},{w})\|_{\dot B^{2}_{2,1}}   \|\q \u\| _{\dot B^{0}_{2,1}}.
\end{align}

On the other hand, since
 $$\big|[\p, (\p \u)\cdot\nabla ]\p \u\big|
\lesssim\big| (\p \u)\cdot\nabla \p \u\big|,$$
 we get from \eqref{Q5} that
\begin{align}\label{QWER5}
&\|[\p, (\p \u)\cdot\nabla ]\p \u \|_{\dot B^{0}_{2,1}}
\lesssim\|{v}\|_{\dot B^{0}_{2,1}}\|{w}\|_{\dot B^{2}_{2,1}}+\|{v}\|_{\dot B^{1}_{2,1}}\|{w}\|_{\dot B^{1}_{2,1}}.
\end{align}
The combination of the  estimates \eqref{QWER2}--\eqref{QWER5} yields
\begin{align}\label{QWER11}
\|[\p, \u\cdot \nabla ]\u   \|_{\dot B^{0}_{2,1}} \lesssim&\|{v}\|_{\dot B^{1}_{2,1}}\|{w}\|_{\dot B^{1}_{2,1}}
  \nonumber\\
  &+(\|\q \u\| _{\dot B^{0}_{2,1}}+\|{v}\|_{\dot B^{0}_{2,1}})(\|\q \u\| _{\dot B^{2}_{2,1}}+\|({v},{w})\| _{\dot B^{2}_{2,1}}).
\end{align}


Inserting \eqref{D12+1}--\eqref{D12+2}  and \eqref{QWER11} into \eqref{D12} gives
\begin{align}\label{D13}
&\|{v}\|_{\widetilde{L}^\infty_t(\dot B^{0}_{2,1})}+\|{v}\|_{L^1_t(\dot B^{2}_{2,1})}\nonumber \\
&\quad\lesssim\|{v}_0\|_{\dot B^{0}_{2,1}}+\int_0^t\|{v}\|_{\dot B^{1}_{2,1}}\|{w}\|_{\dot B^{1}_{2,1}}\,d\tau\nonumber \\
&\quad\quad+\int_0^t\|({v},{w},\q \u) \|_{ \dot B^{2}_{2,1} }(\|a \|_{ \dot B^{1}_{2,1} }+\|({v},\q \u)\|_{\dot B^{0}_{2,1}})\,d\tau.
\end{align}
The combination of \eqref{A19} and \eqref{D13} produces the desired estimate:
\begin{align}\label{E3}
&\|(a,\q \u,{v})(t)\|_{\dot B^{0}_{2,1}}+\|a(t)\|_{\dot B^{  1}_{2,1}}
+\int^t_0\big(\|(a^\ell,\q \u,{v})\|_{\dot B^{2}_{2,1}}+\|a^h\|_{\dot B^{1}_{2,1}}\big)\,d\tau\nonumber\\
&\quad\lesssim\|(a_0,\q \u_0,{v}_0)\|_{\dot B^{0}_{2,1}}+ \|a_0\|_{\dot B^{1}_{2,1}}+\int_0^t\|{v}\|_{\dot B^{1}_{2,1}}\|{w}\|_{\dot B^{1}_{2,1}}\,d\tau\nonumber\\
&\quad\quad+\int_0^t\left(\|a\|_{\dot B^{1}_{2,1}}+\|(a,\q \u,{v})\|_{\dot B^{0}_{2,1}} )(\|a^h\|_{\dot B^{1}_{2,1}}+\|(a^\ell,{v},{w},\q \u)\|_{\dot B^{2}_{2,1}}\right)\,d\tau.
\end{align}

The global solution is to be deduced  by a   continuity argument based on the estimates obtained.
Indeed, let $0< c_1\ll1$ be a sufficient small   constant, which will be determined later on. Define
\begin{align}\label{E5}
T^{**}\stackrel{\mathrm{def}}{=}\sup \Bigg\{ t\in [0,T^*): &\|(a,\q\u,{v})\|_{\wt{L}^\infty_t(\dot{B}^{0}_{2,1})} +\|a\|_{\wt{L}^\infty_t(\dot{B}^{ 1}_{2,1})}\nonumber\\
&\quad+\|(a^\ell,{v},\q\u)\|_{{L}^1_t(\dot{B}^{ 2}_{2,1})}  +\|a^h\|_{{L}^1_t(\dot{B}^{ 1}_{2,1})}\leq c_1 \Bigg\}.
\end{align}
According to the local wellposedness for the system, it is obvious that $T^{**}>0.$ We shall prove $T^{**}=\infty$ under the assumption \eqref{tiaojian}.

By using the interpolation inequality, we can get
\begin{align}\label{E5+we}
\int_0^t\|{v}\|_{\dot B^{1}_{2,1}}\|{w}\|_{\dot B^{1}_{2,1}}\,d\tau
\le& C\int_0^t\|{v}\|_{\dot B^{0}_{2,1}}^{\frac12}\|{v}\|_{\dot B^{2}_{2,1}}^{\frac12}\|{w}\|_{\dot B^{1}_{2,1}}\,d\tau\nonumber\\
\le& \frac18\int_0^t\|{v}\|_{\dot B^{2}_{2,1}}\,d\tau+C\int_0^t\|{v}\|_{\dot B^{0}_{2,1}}\|{w}\|_{\dot B^{1}_{2,1}}^2\,d\tau.
\end{align}
Inserting \eqref{E5+we} into \eqref{E3} and
using \eqref{E5}, we get  for  $t\in [0, T^{**}]$ that
\begin{align}\label{E8+12}
&\|(a,\q\u,{v})(t)\|_{\dot B^{0}_{2,1}}+\|a(t)\|_{\dot B^{  1}_{2,1}}
+\f12\int^t_0\big(\|(a^\ell,\q\u,{v})\|_{\dot B^{2}_{2,1}}+\|a^h\|_{\dot B^{1}_{2,1}}\big)\,d\tau\nonumber\\
&\quad\lesssim\|(a_0,\q\u_0,{v}_0)\|_{\dot B^{0}_{2,1}}+ \|a_0\|_{\dot B^{1}_{2,1}}\nonumber\\
&\quad\quad+\int_0^t(\|{w}\|_{\dot B^{2}_{2,1}}+\|{w}\|_{\dot B^{1}_{2,1}}^2)(\|a\|_{\dot B^{1}_{2,1}}+\|(a,\q\u,{v})\|_{\dot B^{0}_{2,1}} )\,d\tau
\end{align}
which and the Gronwall inequality imply that
\begin{align}\label{E9}
&\|(a,\q\u,{v})(t)\|_{\dot B^{0}_{2,1}}+\|a(t)\|_{\dot B^{  1}_{2,1}}
+\f12\int^t_0\big(\|(a^\ell,\q\u,{v})\|_{\dot B^{2}_{2,1}}+\|a^h\|_{\dot B^{1}_{2,1}}\big)\,d\tau\nonumber\\
&\quad\lesssim\big(\|(a_0,\q\u_0,{v}_0)\|_{\dot B^{0}_{2,1}}+ \|a_0\|_{\dot B^{1}_{2,1}}\big)\exp \big( \int^t_0(\|{w}\|_{\dot B^{2}_{2,1}}+\|{w}\|_{\dot B^{1}_{2,1}}^2)\,d\tau\big)
\end{align}
for  $t\in [0, T^{**}].$

Next, we shall show that the integration $\int^t_0(\|{w}\|_{\dot B^{2}_{2,1}}+\|{w}\|_{\dot B^{1}_{2,1}}^2)\,d\tau$ on the right hand side of \eqref{E9} can  be controlled by the initial data if there holds \eqref{E5}.
Indeed, along the same lines as the derivation of \eqref{D12}, we can infer from the second equation of \eqref{m2} that
\begin{align}\label{D12678}
&\|{w}\|_{\widetilde{L}^\infty_t(\dot B^{0}_{2,1})}+\|{w}\|_{L^1_t(\dot B^{2}_{2,1})}+\|{w}\|_{L^2_t(\dot B^{1}_{2,1})}\nonumber \\
&\quad\lesssim\|{w}_0\|_{\dot B^{0}_{2,1}}+ \int_0^t\|\u\cdot\nabla w\|_{ \dot B^{0}_{2,1} }\,d\tau\nonumber \\
&\quad\quad+\int_0^t\|\p (L(a)\Delta \u+L(a)\nabla\div \u)\|_{ \dot B^{0}_{2,1} }\,d\tau +\int_0^t\|[\p,\u\cdot \nabla] \u\|_{ \dot B^{0}_{2,1} }\,d\tau.
\end{align}
The last two terms on the right hand side of \eqref{D12678} have been bounded by \eqref{D12+2} and \eqref{QWER11}. To control the second term, we have by using the fact $\partial_1 v+\partial_2 w=0$ and product law in Besov spaces that
\begin{align}\label{mengrao}
\|\u\cdot\nabla w\|_{ \dot B^{0}_{2,1} }
\lesssim&\|v\partial_1 w\|_{ \dot B^{0}_{2,1} }+\|w\partial_1 v\|_{ \dot B^{0}_{2,1} }+\|\q\u\cdot\nabla w\|_{ \dot B^{0}_{2,1} }\nonumber \\
\lesssim&\|{v}\|_{\dot B^{1}_{2,1}}\|{w}\|_{\dot B^{1}_{2,1}}+\|\q\u\|_{\dot B^{0}_{2,1}}\|{w}\|_{\dot B^{2}_{2,1}}.
\end{align}

Plugging the estimates \eqref{D12+2},  \eqref{QWER11} and \eqref{mengrao} into \eqref{D12678} gives
\begin{align}\label{E1}
&\|{w}\|_{\widetilde{L}^\infty_t(\dot B^{0}_{2,1})}+\|{w}\|_{L^1_t(\dot B^{2}_{2,1})}+\|{w}\|_{L^2_t(\dot B^{1}_{2,1})}\nonumber \\
&\quad\lesssim\|{w}_0\|_{\dot B^{0}_{2,1}}+\int_0^t\|{v}\|_{\dot B^{1}_{2,1}}\|{w}\|_{\dot B^{1}_{2,1}}\,d\tau\nonumber \\
&\quad\quad+\int_0^t\|({v},{w},\q \u) \|_{ \dot B^{2}_{2,1} }(\|a \|_{ \dot B^{1}_{2,1} }+\|({v},\q \u)\|_{\dot B^{0}_{2,1}})\,d\tau.
\end{align}
Employing the interpolation inequality, we see that
\begin{align}\label{E5+we+ni}
\int_0^t\|{v}\|_{\dot B^{1}_{2,1}}\|{w}\|_{\dot B^{1}_{2,1}}\,d\tau
&\lesssim\int_0^t\|{v}\|_{\dot B^{0}_{2,1}}^{\frac12}\|{v}\|_{\dot B^{2}_{2,1}}^{\frac12}\|{w}\|_{\dot B^{1}_{2,1}}\,d\tau\nonumber\\
&\lesssim\|{v}\|_{\wt L^\infty_t(\dot B^{0}_{2,1})}\|{v}\|_{L^1_t(\dot B^{2}_{2,1})}\|{w}\|_{ L^2_t(\dot B^{1}_{2,1})}.
\end{align}

Hence,
inserting \eqref{E5+we+ni} into  \eqref{E1} and using \eqref{E5}, we can get
for  $t\in [0, T^{**}]$ that
\begin{align}\label{E12new}
&\|{w}\|_{\wt L^\infty_t(\dot B^{0}_{2,1})}+\|{w}\|_{L^1_t(\dot B^{2}_{2,1})}+\|{w}\|_{L^2_t(\dot B^{1}_{2,1})}\le C(\|(a_0,\q\u_0,{v}_0,{w}_0)\|_{\dot B^{0}_{2,1}}+ \|a_0\|_{\dot B^{1}_{2,1}}).
\end{align}
Taking the above inequality into \eqref{E9}, we can finally get
 \begin{align}\label{E15}
&\|(a,\q\u,{v})(t)\|_{\dot B^{0}_{2,1}}+\|a(t)\|_{\dot B^{  1}_{2,1}}
+\f12\int^t_0\big(\|(a^\ell,\q\u,{v})\|_{\dot B^{2}_{2,1}}+\|a^h\|_{\dot B^{1}_{2,1}}\big)\,d\tau\nonumber\\
&\quad\le C
(\|(a_0,\q\u_0,{v}_0)\|_{\dot B^{0}_{2,1}}\!\!+ \|a_0\|_{\dot B^{1}_{2,1}})
\exp\big( C(\|(a_0,\q\u_0,{v}_0,{w}_0)\|_{\dot B^{0}_{2,1}}\!\!+ \|a_0\|_{\dot B^{1}_{2,1}})^2\big).
\end{align}

Thus,
if the smallness condition \eqref{tiaojian} is satisfied, then  \eqref{E15} implies that
\begin{align*}
\|(a,\q\u,{v})\|_{\wt{L}^\infty_t(\dot{B}^{0}_{2,1})} +\|a\|_{\wt{L}^\infty_t(\dot{B}^{ 1}_{2,1})}
+\|(a^\ell,{v},\q\u)\|_{{L}^1_t(\dot{B}^{ 2}_{2,1})}  +\|a^h\|_{{L}^1_t(\dot{B}^{ 1}_{2,1})}   \leq \f{c_1}{2}
\end{align*}
for $t\leq T^{**}$. This   contradicts to   \eqref{E5}.
Hence  we conclude that $T^{**}=\infty$ and the proof of Theorem \ref{dingli} is complete.


\noindent \textbf{Acknowledgement.} This work is supported by NSFC under grant numbers 11601533 and 11571240.




\noindent
\begin{thebibliography}{10}
\bibitem{bcd}
H.~Bahouri, J.~Y. Chemin, R.~Danchin,
\newblock { {F}ourier {A}nalysis and {N}onlinear {P}artial {D}ifferential
  {E}quations}.
\newblock Grundlehren Math. Wiss., Vol. {\textbf{343}}, Springer-Verlag,
  Berlin, Heidelberg, 2011.
\newblock $\,$

\bibitem{charve} F. Charve, R. Danchin,  A global existence result for the compressible Navier-Stokes equations in the critical $L^p$ framework,  {\it Arch. Ration. Mech. Anal.},  {\bf198} (2010), 233--271.


\bibitem{chenqionglei} Q. Chen, C. Miao,  Z. Zhang,  Global well-posedness for compressible Navier-Stokes
equations with highly oscillating initial velocity, {\it Comm. Pure Appl. Math.}, {\bf 63} (2010), 1173--1224.
\bibitem{zhaixiaoping}
 Z.~Chen, X.~Zhai,
\newblock
 Global large solutions and incompressible limit for the compressible Navier-Stokes equations,
\newblock  {\it J. Math. Fluid Mech.}, {\bf 21} (2019), Art. 26, 23.
\newblock $\,$


\bibitem{danchin2000} R. Danchin,
 Global existence  in  critical spaces for compressible Navier-Stokes equations,
{\it Invent. Math.}, {\bf 141} (2000), 579--614.

\bibitem{danchin2014} R. Danchin,
 A Lagrangian approach for the compressible Navier-Stokes equations,
{\it Annales de l'Institut Fourier}, {\bf 64} (2014), 753--791.







\bibitem{danchin2018}
R.~Danchin, P. Mucha,
Compressible Navier-Stokes system: large solutions and incompressible limit,
 {\it Adv. Math.}, {\bf 320} (2017), 904--925.



\bibitem{fangdaoyuan} D. Fang, T. Zhang, R. Zi,  Global solutions to the isentropic compressible Navier-Stokes equations with a class of large initial data,   {\it SIAM J. Math. Anal.}, {\bf 50} (2018), 4983--5026.




\bibitem{Feireisl1}
E. Feireisl, {Dynamics of Viscous Compressible Fluids}. { Oxford Univ. Press, Oxford,} 2004.



\bibitem{ferisal} E. Feireisl, A. Novotn\'{y}, H. Petzeltov\'{a}, On the global existence of globally defined
weak solutions to the Navier-Stokes equations of isentropic compressible fluids,  {\it J. Math. Fluid Mech.},  {\bf3} (2001),
358-392.

\bibitem{ferisal2} E. Feireisl, P. Gwiazda, A. \'{S}wierczewska-Gwiazda,  E. Wiedemann,  Dissipative measure-valued
solutions to the compressible Navier-Stokes system, {\it Calc. Var. Partial Differ.},  {\bf55} (2016),
55--141.

\bibitem{Feireisl4} E. Feireisl, A. Novotn\'{y},  Y. Sun, Suitable weak solutions to the Navier-Stokes equations
of compressible viscous fluids, {\it Indiana Univ. Math. J.},  {\bf60} (2011),
611--631.

\bibitem{haspot} B. Haspot,  Existence of global strong solutions in critical spaces for barotropic viscous fluids,
{\it  Arch. Ration. Mech. Anal.}, {\bf 202} (2011), 427--460.


\bibitem{huangjingchi2} L. He, J. Huang, C. Wang,  Global stability of large solutions to the 3D compressible Navier-Stokes equations,
{\it  Arch. Rational Mech. Anal.}, {\bf 234} (2019), 1167--1222.





\bibitem{hoff1995} D. Hoff,  Global solutions of the Navier-Stokes equations for multidimensional compressible
flow with discontinuous initial data,  {\it J. Differential Equations}, {\bf 120} (1995), 215--254.






\bibitem{huangjingchi}
J. Huang, M. Paicu, P.  Zhang,  Global well-posedness of incompressible inhomogeneous fluid
              systems with bounded density or non-{L}ipschitz velocity,
 {\it Arch. Ration. Mech. Anal.}, {\bf{209}} (2013), 631--682.




\bibitem{huangxiangdi}
X. Huang, J. Li, Z.  Xin,  Global well-posedness of classical
solutions with large oscillations and vacuum to the
three-dimensional isentropic compressible Navier-Stokes equaitons,
 {\it Commun. Pure Appl. Math.}, {\bf{65}} (2012), 549--585.


\bibitem{Kotschote} M. Kotschote, Dynamical stability of non-constant equilibria for the compressible Navier-Stokes equations in Eulerian coordinates,  {\it  Comm. Math. Phys.}, {\bf 328},  (2014), 809--847.


\bibitem{lions+1} P.L. Lions, { Mathematical Topics in Fluid Mechanics}. Vol.2, Compressible models,
Oxford University Press, 1998.




\bibitem{mat}  A. Matsumura, T. Nishida,  The initial value problem for the equations of motion of compressible viscous and heat-conductive fluids,
 {\it Proc. Japan Acad. Ser. A Math. Sci.},  {\bf55} (1979), 337--342.



\bibitem{paicu2012}
M.~Paicu,  P.~Zhang,
\newblock  Global solutions to the 3-{D} incompressible inhomogeneous
  {N}avier-{S}tokes system,
\newblock  {\it J. Funct. Anal.}, {\bf 262} (2012), 3556--3584.
\newblock $\,$

\bibitem{villani} C. Villani,  Hypocoercivity, {\it Mem. Amer. Math. Soc.}, {\bf202} (2009), no. 950.

\bibitem{wangchao}
C.~Wang, W. Wang,  Z.~Zhang,
\newblock  Global well-posedness of compressible Navier-Stokes equations for some classes of
large initial data,
\newblock  {\it Arch. Ration. Mech. Anal.}, {\bf 213} (2014), 171--214.
\newblock $\,$


\bibitem{Xin} Z. Xin,  Blowup of smooth solutions to the compressible Navier-Stokes equation with compact density,
{\it Comm. Pure Appl. Math.},  {\bf51} (1998), 229--240.








\end{thebibliography}
\end{document}